\documentclass[10pt, a4paper]{article}

\usepackage{cite, amsmath, amssymb, booktabs, lastpage, graphicx}
\usepackage[hidelinks]{hyperref}

\usepackage{geometry}
\usepackage{setspace}

\usepackage[mathscr]{euscript}
\usepackage{amsthm}
\usepackage{xcolor}
\theoremstyle{plain}
\newtheorem{lemma}{Lemma}
\newtheorem{theorem}{Theorem}

\newtheorem*{corollary}{Corollary}

\theoremstyle{remark}

\theoremstyle{definition}
\newtheorem{definition}{Definition}
\newtheorem{example}{Example}

\usepackage{graphicx}

\makeatletter
\renewcommand{\maketitle}{
\begin{center}

{\Large\bfseries \@title\par}
\vspace{6mm}

{\large\bfseries \@author\par}
\vspace{4mm}

{\itshape \@address\par}
\vspace{2mm}

{\small\ttfamily \@email\par}
\vspace{5mm}

\vspace{5mm}

\end{center}
}
\makeatother

\makeatletter
\newcommand{\address}[1]{\gdef\@address{#1}}
\newcommand{\email}[1]{\gdef\@email{#1}}
\makeatother

\address{}
\email{}

\usepackage{fancyhdr}

\usepackage[margin=1cm,%
			font=footnotesize,%
			format=hang,%
			labelsep=period,%
			labelfont=bf]{caption}
\allowdisplaybreaks

\title{Topological Indices of Divisor Prime Graphs}
\author{Purva J. Makadiya $^{a,}$\footnote{Corresponding author.}, Mahesh M. Jariya$^b$, Prashant J. Makadiya$^c$}
\address{$^{a,b}$Department of Mathematics, \\ Saurashtra University,  Rajkot-360005, \\ Gujarat, India.
    }
\email{purvamakadiya2000@gmail.com, mahesh.jariya@gmail.com, prashantmakadiya1996@gmail.com}

\date{\today}

\begin{document}

\maketitle
\thispagestyle{empty}

\begin{abstract}
Graph theory provides powerful tools for modeling concepts in number theory, leading to the introduction of graphs derived from arithmetic properties. One such structure is the divisor prime graph, $G_{Dp(n)}$. For any positive integer $n$, let $D(n)$ be the set of its positive divisors. The vertex set of $G_{Dp(n)}$ consists of the elements of $D(n)$, with the adjacency condition that two vertices $x$ and $y$ share an edge if and only if their greatest common divisor is $1$. The primary focus of this study is to evaluate the topological characteristics of $G_{Dp(n)}$. To achieve this, we analyze and compute various distance and degree-based indices, specifically focusing on the Wiener, Harary, hyper-Wiener, First and Second Zagreb, Schultz, Gutman, and Eccentric connectivity indices.
\end{abstract}

\vspace{2mm}

\vspace{1mm}

\onehalfspacing
\section{Introduction}
Applying graph theory to number theory has led to important discoveries about mathematical structures. Creating graphs from sets of numbers is a reliable method for uncovering hidden features. By translating math relationships into visual networks, we can study abstract systems in new and effective ways.

In parallel, graph theory has become a cornerstone of theoretical chemistry. Ever since Harry Wiener introduced the Wiener index \cite{wiener1947structural} in 1947 to predict boiling points, scientists have used "topological indices" to connect molecular properties with their structural graphs. These mathematical metrics serve as useful tools to measure the distance, branching, and symmetry within any network.

Current research is bringing these two areas together by computing topological indices for number-based graphs. A perfect example of this connection is the \textit{divisor prime graph} defined as follows:
\begin{definition}\cite{nair2022divisor} For any positive integer $n \ge 1$ with $r$ divisors $d_1, d_2, \dots, d_r$, the \textbf{divisor prime graph} $G_{Dp(n)}$ is a graph with vertex set $V = \{d_1, d_2, \dots, d_r\}$ such that two distinct vertices $d_i$ and $d_j$ are adjacent if and only if $\gcd(d_i, d_j) = 1$. Since we consider only simple graphs, the possible loop at vertex $1$ is neglected.
\end{definition}

In this paper, we systematically analyze the algebraic structure defined by the prime factorization of $n$ to compute exact, generalized closed-form expressions for a wide array of topological descriptors of $G_{Dp(n)}$. Specifically, we derive the exact formulas for distance-based invariants (Wiener, Harary, hyper-Wiener indices), degree-based invariants (First and Second Zagreb indices), and mixed invariants (Schultz, Gutman, and Eccentric Connectivity indices). 

\section{Preliminaries}
In this section, we introduce the fundamental notations, graph-theoretic terminology, and structural observations concerning the divisor prime graph that will be utilized throughout the proofs in this paper. 

\vspace{2mm}
\noindent \textbf{Basic Graph Theory Notations:}
Let $G = (V, E)$ be a simple, undirected, and connected graph, where $V(G)$ represents the vertex set and $E(G)$ represents the edge set. The number of vertices, $|V(G)|$, is called the order of $G$, and the number of edges, $|E(G)|$, is called the size of $G$. 

For any vertex $v \in V(G)$, the degree of $v$, denoted by $d(v)$, is the number of edges incident to $v$. The distance between two vertices $u$ and $v$, denoted by $d(u,v)$, is the length of the shortest path connecting them in $G$. The eccentricity of a vertex $v$, denoted by $\varepsilon(v)$, is the maximum distance between $v$ and any other vertex in $G$, mathematically defined as $\varepsilon(v) = \max_{u \in V(G)} d(v, u)$. 

\vspace{2mm}
\noindent \textbf{Structure of $G_{Dp(n)}$:}
Let the prime factorization of $n$ be $n = p_1^{k_1} p_2^{k_2} \dots p_r^{k_r}$, where $p_i$ are distinct prime numbers and $k_i \ge 1$ are positive integers. The total number of divisors of $n$, which corresponds to the total number of vertices in $G_{Dp(n)}$, is given by the divisor function:
$$|V| = D = \prod_{i=1}^r (k_i + 1).$$
Because the number $1$ is a divisor of every positive integer, it is always a vertex in $G_{Dp(n)}$. Furthermore, since $\gcd(1, u) = 1$ for any positive integer $u$, the vertex $1$ is adjacent to every other vertex in the graph. We refer to $1$ as the \textit{central vertex}.

\subsection{Distance Based Topological Indices}

The first use of a topological index was made in 1947 by chemist Wiener. The \textbf{Wiener index} \cite{wiener1947structural} is the oldest topological index, denoted by $W(G)$, which is defined as the sum of distances between every pair of vertices of $G$, i.e.,
$$W(G) = \sum_{u,v \in V} d(u,v),$$
where $d(u,v)$ is the length of the shortest distance between $u$ and $v$. 
\\The \textbf{Harary index} \cite{plavsic1993indices} was introduced independently in 1993 by Plavsi\'{c} et al. The index was named in honor of Professor Frank Harary on the occasion of his $70^{th}$ birthday. Originally developed for the characterization of molecular graphs, the Harary index is derived from the reciprocal distance matrix and is known to possess significant physical and chemical applications.
It is calculated for a connected graph $G$ by the formula:
$$H(G) = \sum_{u,v \in V} \frac{1}{d(u, v)}.$$
Where, $d(u,v)$ denotes the shortest-path distance between vertices $u$ and $v$ within the graph. 
\\Later, Randić introduced the \textbf{hyper-Wiener index} \cite{randic1993novel} to provide a more sensitive measure of molecular branching. It incorporates both the distances and the squares of the distances:
$$WW(G) = \frac{1}{2} \sum_{u,v \in V} (d(u,v) + d(u,v)^2).$$
\subsection{Degree-Based Topological Indices}
While distance-based indices capture the overall spread of a graph, degree-based indices focus on local connectivity. Introduced by Gutman and Trinajstić in 1972 to calculate the total $\pi$-electron energy of alternant hydrocarbons, the Zagreb indices are among the most important degree-based invariants.
The \textbf{First Zagreb index} \cite{cvetkovic1972graph} is defined as the sum of the squares of the degrees of all vertices:
$$M_1(G) = \sum_{v \in V} d(v)^2 = \sum_{uv \in E} (d(u) + d(v)).$$
The \textbf{Second Zagreb index} \cite{cvetkovic1972graph} is defined as the sum of the products of the degrees of pairs of adjacent vertices:
$$M_2(G) = \sum_{u,v \in V} d(u)d(v).$$
\subsection{Mixed Topological Indices}
To capture a more holistic topological profile, several indices have been developed that combine both distance and degree properties. \textbf{Schultz index} (or Molecular Topological index) \cite{Schultz1989}, introduced in 1989, weighs the sum of the degrees of two vertices by the distance between them:
        $$S(G) = \sum_{u,v \in V} (d(u) + d(v)) d(u,v).$$
Similarly, the \textbf{Gutman index} \cite{Gutman1994} (sometimes referred to as the Schultz index of the second kind) weighs the product of the degrees by the distance:
$$Gut(G) = \sum_{u,v \in V} d(u)d(v) d(u,v).$$
Finally, the \textbf{Eccentric Connectivity index} \cite{Sharma1997} integrates the degree of a vertex with its eccentricity, providing a highly discriminatory measure of molecular structure:
$$\xi^c(G) = \sum_{v \in V} d(v) \varepsilon(v).$$
\begin{lemma} \label{dia_lemma}\cite{kalita2025properties}  For all $n \in \mathbb{Z}^+$, $\text{diam}(G_{Dp(n)}) \le 2$.
\end{lemma} 
\noindent \textbf{Basic Observations:}
The proofs of our main topological indices rely heavily on the specific edge counts and distance distributions within $G_{Dp(n)}$. We establish the following structural properties as observations.

\noindent \textbf{Observation (Vertex Degrees).} Let $n = p_1^{k_1} p_2^{k_2} \dots p_r^{k_r}$. For any vertex $u \in V(G_{Dp(n)})$, let $S_u$ be the set of indices corresponding to the prime factors present in the factorization of $u$. The degree of $u$ in $G_{Dp(n)}$ is:
$$d(u) = \prod_{i \notin S_u} (k_i + 1).$$
Specifically, for the central vertex $1$, $S_1 = \emptyset$, yielding $d(1) = D - 1$.

\noindent \textbf{Observation (Total Edges).}
The total number of edges in the divisor prime graph $G_{Dp(n)}$ is given by:
$$|E| = \frac{1}{2} \left( \prod_{i=1}^r (2k_i + 1) - 1 \right).$$

\begin{proof}
To find $|E|$, we count the total number of ordered pairs $(u, v)$ of divisors such that $\gcd(u, v) = 1$. 
Any divisors $u$ and $v$ can be uniquely written as $u = \prod_{i=1}^r p_i^{a_i}$ and $v = \prod_{i=1}^r p_i^{b_i}$, where $0 \le a_i \le k_i$ and $0 \le b_i \le k_i$. 
The condition $\gcd(u, v) = 1$ is satisfied if and only if for every prime factor $p_i$, at least one of the exponents $a_i$ or $b_i$ is zero (i.e., $\min(a_i, b_i) = 0$). 
For a fixed prime $p_i$, the number of valid pairs $(a_i, b_i)$ is determined as follows:
\begin{itemize}
    \item If $a_i = 0$, $b_i$ can take any of the $k_i + 1$ values from $\{0, 1, \dots, k_i\}$.
    \item If $b_i = 0$, $a_i$ can take any of the $k_i$ non-zero values from $\{1, 2, \dots, k_i\}$ (to avoid double-counting the $a_i=0, b_i=0$ case).
\end{itemize}
Thus, there are $(k_i + 1) + k_i = 2k_i + 1$ valid pairs of exponents for each prime $p_i$. Since the choices for each prime factor are independent, the total number of ordered pairs $(u, v)$ such that $\gcd(u,v) = 1$ is $\prod_{i=1}^r (2k_i + 1)$.

This total count includes:
\begin{enumerate}
    \item Pairs where $u = v$. Since $\gcd(u, u) = u$, this is only satisfied when $u = 1$, contributing exactly $1$ pair.
    \item Pairs where $u \neq v$. Each undirected edge $\{u, v\}$ corresponds to exactly $2$ ordered pairs: $(u, v)$ and $(v, u)$, contributing $2|E|$ pairs.
\end{enumerate}
Equating our counts, we have $\prod_{i=1}^r (2k_i + 1) = 1 + 2|E|$. Solving for $|E|$ yields the desired result.
\end{proof}
These preliminary definitions and observations establish the exact combinatorial parameters required to compute the distance and degree-based topological indices in the subsequent sections.

\section{Main Results}

\begin{theorem} \label{thmwiener}
Let $n = p_1^{k_1} p_2^{k_2} \dots p_r^{k_r}$ be the prime factorization of a positive integer $n$, where $p_i$ are distinct primes and $k_i \ge 1$. Let $D = \prod_{i=1}^r (k_i + 1)$ be the total number of divisors of $n$. The Wiener index of the divisor prime graph $G_{Dp(n)}$ is given by:
$$W(G_{Dp(n)}) = D(D - 1) - \frac{1}{2} \left( \prod_{i=1}^r (2k_i + 1) - 1 \right).$$
\end{theorem}

\begin{proof}
Let $V$ be the vertex set of $G_{Dp(n)}$. By definition, the vertices are the divisors of $n$, yielding a total of $|V| = D = \prod_{i=1}^r (k_i + 1)$ vertices. 

By Lemma \ref{dia_lemma}, the diameter of $G_{Dp(n)}$ is at most 2. Therefore, for any unordered pair of distinct vertices $\{u, v\}$, the shortest path distance $d(u,v)$ is exactly $1$ if they form an edge, and exactly $2$ if they do not. 

Let $|E|$ represent the total number of edges in $G_{Dp(n)}$. The Wiener index $W(G)$ can be calculated by grouping the $\binom{D}{2}$ total pairs by their distances:\
\begin{align*}
W(G_{Dp(n)})& = \sum_{uv \in E} d(u,v) + \sum_{uv \notin E} d(u,v)\\
& = 1 \cdot |E| + 2 \cdot \left( \binom{D}{2} - |E| \right)\\
&= 2 \binom{D}{2} - |E| = D(D-1) - |E|.
\end{align*}
From our earlier Observation regarding total edges, we know:
$$|E| = \frac{1}{2} \left( \prod_{i=1}^r (2k_i + 1) - 1 \right).$$
Substituting this explicit value of $|E|$ back into our Wiener index equation yields:
$$W(G_{Dp(n)}) = D(D - 1) - \frac{1}{2} \left( \prod_{i=1}^r (2k_i + 1) - 1 \right).$$
This completes the proof.
\end{proof}

\begin{example}Wiener index for the divisor prime graph of $n = 12$.
\end{example}

For $n = 12 = 2^2 \cdot 3^1$, the vertex set consists of $D = (2+1)(1+1) = 6$ divisors: $V = \{1, 2, 3, 4, 6, 12\}$. Two vertices are adjacent if they are relatively prime. The central vertex $1$ connects to all $5$ other vertices. Among the remaining vertices, only the pairs $\{2,3\}$ and $\{3,4\}$ satisfy $\gcd(u,v) = 1$. Thus, the graph contains exactly $|E| = 7$ edges (pairs at distance 1).

Because the maximum diameter is 2, the remaining unordered pairs of vertices must be at a distance of 2. Out of the $\binom{6}{2} = 15$ total pairs, $15 - 7 = 8$ pairs are at distance 2.

Summing the distances of all pairs manually yields:
$$W(G_{Dp(12)}) = 7(1) + 8(2) = 23.$$

Applying our newly derived formula in Theorem \ref{thmwiener} confirms this result seamlessly. Substituting $D = 6$ and $|E| = 7$:
$$W(G_{Dp(12)}) = D(D - 1) - |E| = 6(5) - 7 = 30 - 7 = 23.$$

\begin{theorem} \label{thmharary}
Let $n = p_1^{k_1} p_2^{k_2} \dots p_r^{k_r}$ be the prime factorization of a positive integer $n$, where $p_i$ are distinct primes and $k_i \ge 1$. Let $D = \prod_{i=1}^r (k_i + 1)$ be the total number of divisors of $n$. The Harary index of the divisor prime graph $G_{Dp(n)}$ is given by:
$$H(G_{Dp(n)}) = \frac{D(D - 1) + \prod_{i=1}^r (2k_i + 1) - 1}{4}.$$
\end{theorem}

\begin{proof}
Let $V$ be the vertex set of $G_{Dp(n)}$, where $|V| = D = \prod_{i=1}^r (k_i + 1)$. 

By Lemma \ref{dia_lemma}, the diameter of $G_{Dp(n)}$ is at most 2. Thus, the shortest path distance $d(u,v)$ between any two distinct vertices is exactly $1$ if $uv \in E(G_{Dp(n)})$ and $2$ if $uv \notin E(G_{Dp(n)})$. 

The Harary index $H(G)$ partitions the $\binom{D}{2}$ total unordered pairs into edges and non-edges:
$$H(G_{Dp(n)}) = \sum_{uv \in E} \frac{1}{1} + \sum_{uv \notin E} \frac{1}{2} = |E| + \frac{1}{2} \left( \binom{D}{2} - |E| \right).$$
Simplifying this expression yields:
$$H(G_{Dp(n)}) = \frac{1}{2}|E| + \frac{D(D-1)}{4}.$$
By our earlier Observation on Total Edges, we know $$|E| = \frac{1}{2} \left( \prod_{i=1}^r (2k_i + 1) - 1 \right). $$Substituting this exact value into our simplified Harary index equation gives:
$$H(G_{Dp(n)}) = \frac{D(D-1)}{4} + \frac{1}{4} \left( \prod_{i=1}^r (2k_i + 1) - 1 \right).$$
Combining the terms over the common denominator yields the final formula:
$$H(G_{Dp(n)}) = \frac{D(D - 1) + \prod_{i=1}^r (2k_i + 1) - 1}{4}.$$
This completes the proof.
\end{proof}

\begin{example} Harary index for the divisor prime graph of $n = 12$.
\end{example}

For $n = 12 = 2^2 \cdot 3^1$, the graph consists of $D = 6$ vertices. As determined previously, there are exactly $|E| = 7$ edges (pairs at distance 1). Because the maximum diameter is 2, the remaining $\binom{6}{2} - 7 = 15 - 7 = 8$ unordered pairs must be at a distance of 2.

Summing the reciprocal distances manually yields:
$$H(G_{Dp(12)}) = 7\left(\frac{1}{1}\right) + 8\left(\frac{1}{2}\right) = 7 + 4 = 11.$$

We can seamlessly verify this result using the closed-form expression derived in Theorem \ref{thmharary}. Substituting $D = 6$, $k_1 = 2$, and $k_2 = 1$:
\begin{align*}
H(G_{Dp(12)})& = \frac{6(6 - 1) + (2(2) + 1)(2(1) + 1) - 1}{4}\\
& = 11.\\
\end{align*}

\begin{theorem}\label{thmhyperwiener}
Let $n = p_1^{k_1} p_2^{k_2} \dots p_r^{k_r}$ be the prime factorization of a positive integer $n$, where $p_i$ are distinct primes and $k_i \ge 1$. Let $D = \prod_{i=1}^r (k_i + 1)$ be the total number of divisors of $n$. The hyper-Wiener index of the divisor prime graph $G_{Dp(n)}$ is given by:
$$WW(G_{Dp(n)}) = \frac{3D(D-1)}{2} - \prod_{i=1}^r (2k_i + 1) + 1.$$
\end{theorem}

\begin{proof}
Let $V$ be the vertex set of $G_{Dp(n)}$, where $|V| = D = \prod_{i=1}^r (k_i + 1)$. 

By Lemma \ref{dia_lemma}, the diameter of $G_{Dp(n)}$ is at most 2. Thus, the shortest path distance $d(u,v)$ between any two distinct vertices is exactly $1$ if $uv \in E(G_{Dp(n)})$ and $2$ if $uv \notin E(G_{Dp(n)})$. 
The hyper-Wiener index $WW(G)$ is defined as:
$$WW(G) = \frac{1}{2} \sum_{u,v \in V} (d(u,v) + d(u,v)^2).$$
We partition this summation over the $\binom{D}{2}$ total unordered pairs into edges ($d=1$) and non-edges ($d=2$):
\begin{align*}
WW(G_{Dp(n)})& = \sum_{uv \in E} \frac{1}{2}(1 + 1^2) + \sum_{uv \notin E} \frac{1}{2}(2 + 2^2)\\
& = \sum_{uv \in E} 1 + \sum_{uv \notin E} 3\\ & = |E| + 3 \left( \binom{D}{2} - |E| \right).
\end{align*}
Simplifying this algebraic expression yields:
$$WW(G_{Dp(n)}) = 3\binom{D}{2} - 2|E| = \frac{3D(D-1)}{2} - 2|E|.$$
From our earlier Observation on Total Edges, we know $$2|E| = \prod_{i=1}^r (2k_i + 1) - 1.$$ 
Substituting this exact value into our simplified equation gives:
\begin{align*}
WW(G_{Dp(n)})& = \frac{3D(D-1)}{2} - \left( \prod_{i=1}^r (2k_i + 1) - 1 \right)\\
& = \frac{3D(D-1)}{2} - \prod_{i=1}^r (2k_i + 1) + 1.
\end{align*}
This completes the proof.
\end{proof}

\begin{example} Hyper-Wiener index for the divisor prime graph of $n = 15$.
\end{example}

For $n = 15 = 3^1 \cdot 5^1$, the vertex set consists of $D = (1+1)(1+1) = 4$ divisors: $V = \{1, 3, 5, 15\}$. 

Two vertices are adjacent if they are relatively prime. The central vertex $1$ connects to the other $3$ vertices. Among the remaining vertices, only $\{3,5\}$ satisfies $\gcd(u,v) = 1$. Thus, the graph contains exactly $|E| = 4$ edges (pairs at distance 1). Because the maximum diameter is 2, the remaining $\binom{4}{2} - 4 = 6 - 4 = 2$ unordered pairs must be at a distance of 2.

Evaluating the hyper-Wiener index summation manually, pairs at distance 1 contribute $\frac{1}{2}(1+1^2) = 1$, and pairs at distance 2 contribute $\frac{1}{2}(2+2^2) = 3$. Summing these yields:
$$WW(G_{Dp(15)}) = 4(1) + 2(3) = 4 + 6 = 10.$$

We can seamlessly verify this result using the closed-form expression derived in Theorem \ref{thmhyperwiener}. Substituting $D = 4$, $k_1 = 1$, and $k_2 = 1$:
\begin{align*}
WW(G_{Dp(15)})& = \frac{3(4)(4 - 1)}{2} - (2(1) + 1)(2(1) + 1) + 1\\
& = \frac{36}{2} - (3)(3) + 1\\
& = 10.
\end{align*}

\begin{theorem} \label{thmzagreb} Let $n = p_1^{k_1} p_2^{k_2} \dots p_r^{k_r}$ be the prime factorization of a positive integer $n$, where $p_i$ are distinct primes and $k_i \ge 1$. Let $D = \prod_{i=1}^r (k_i + 1)$ be the total number of divisors of $n$. The First and Second Zagreb indices of the divisor prime graph $G_{Dp(n)}$ are given by:
$$M_1(G_{Dp(n)}) = \prod_{i=1}^r \left( (k_i + 1)^2 + k_i \right) - 2D + 1$$
$$M_2(G_{Dp(n)}) = \frac{1}{2} \left[ D \prod_{i=1}^r (3k_i + 1) - 2 \prod_{i=1}^r (2k_i + 1) - D^2 + 2D \right].$$
\end{theorem}

\begin{proof}
Let $V$ be the vertex set of $G_{Dp(n)}$, where $|V| = D$. By our earlier Observation on Vertex Degrees, the degree of the central vertex is $d(1) = D - 1$. For any $u > 1$, its degree is $d(u) = \prod_{i \notin S_u} (k_i + 1)$, where $S_u$ is the set of prime factors of $u$. 

To unify the summation, we define a weight function $w(u) = \prod_{i \notin S_u} (k_i + 1)$ for all $u \in V$. Note that $w(1) = D$, and $w(u) = d(u)$ for $u > 1$.

\vspace{2mm}
\noindent \textbf{First Zagreb index, $M_1(G)$:} \\
The index is $M_1(G) = \sum_{u \in V} d(u)^2$. Substituting our weight function $w(u)$ yields:
\begin{align*}
M_1(G_{Dp(n)}) &= d(1)^2 + \sum_{u > 1} w(u)^2\\ 
& = (D-1)^2 - D^2 + \sum_{u \in V} w(u)^2\\
& = -2D + 1 + \sum_{u \in V} w(u)^2.
\end{align*}
To evaluate the sum over all $V$, we consider the independent choices of exponents for each prime $p_i$ in $u$. If $p_i$ is absent ($1$ choice), it contributes a factor of $(k_i+1)^2$. If $p_i$ is present ($k_i$ choices), it contributes $1^2 = 1$.
$$\sum_{u \in V} w(u)^2 = \prod_{i=1}^r \left( 1 \cdot (k_i+1)^2 + k_i \cdot 1 \right) = \prod_{i=1}^r \left( (k_i+1)^2 + k_i \right).$$
Substituting this back establishes the final formula:
$$M_1(G_{Dp(n)}) = \prod_{i=1}^r \left( (k_i + 1)^2 + k_i \right) - 2D + 1.$$
\vspace{2mm}
\noindent \textbf{Second Zagreb index, $M_2(G)$:} \\
The index $M_2(G) = \sum_{uv \in E} d(u)d(v)$ is partitioned into edges incident to the central vertex $1$ and edges between non-central vertices:
$$M_2(G_{Dp(n)}) = d(1) \sum_{u > 1} d(u) +\sum_{\substack{u>1, v>1 \\ \gcd(u,v)=1}} d(u)d(v).$$
Using the same combinatorial technique, the sum of all weights is $$\sum_{u \in V} w(u) = \prod_{i=1}^r (2k_i + 1).$$ Thus, the non-central degree sum is $$\sum_{u > 1} d(u) = \prod_{i=1}^r (2k_i + 1) - D.$$
Therefore, $$d(1) \sum_{u > 1} d(u) = (D-1) \left( \prod_{i=1}^r (2k_i + 1) - D \right).$$

\noindent Next, we evaluate the unconstrained product sum for all coprime pairs: $$S = \sum_{\gcd(u,v)=1} w(u)w(v).$$ For each prime $p_i$, the valid exponent pairs $(a_i, b_i)$ are those where $\min(a_i, b_i) = 0$.
\begin{itemize}
    \item $a_i = 0, b_i = 0$ (1 pair): contributes $(k_i+1)(k_i+1) = (k_i+1)^2$
    \item $a_i = 0, b_i > 0$ ($k_i$ pairs): contributes $(k_i+1)(1) = k_i+1$
    \item $a_i > 0, b_i = 0$ ($k_i$ pairs): contributes $(1)(k_i+1) = k_i+1$
\end{itemize}
The sum of these contributions is $(k_i+1)^2 + 2k_i(k_i+1) = (k_i+1)(3k_i+1)$. Thus:
$$S = \prod_{i=1}^r (k_i+1)(3k_i+1) = D \prod_{i=1}^r (3k_i+1).$$
We isolate the non-central edges by subtracting terms involving the central vertex $1$ from $S$:
$$S = w(1)^2 + 2w(1)\sum_{u > 1} w(u) + 2 \sum_{\substack{u>1, v>1 \\ \gcd(u,v)=1}} d(u)d(v).$$
$$2 \sum_{\substack{u>1, v>1 \\ \gcd(u,v)=1}} d(u)d(v) = D \prod_{i=1}^r (3k_i+1) - D^2 - 2D \left( \prod_{i=1}^r (2k_i + 1) - D \right).$$
Dividing this by $2$ and adding the central edge contribution $(D-1)\left(\prod_{i=1}^r (2k_i + 1) - D\right)$ yields the simplified formula:
$$M_2(G_{Dp(n)}) = \frac{1}{2} \left[ D \prod_{i=1}^r (3k_i + 1) - 2 \prod_{i=1}^r (2k_i + 1) - D^2 + 2D \right].$$
This completes the proof.
\end{proof}

\begin{example} First and Second Zagreb indices for the divisor prime graph of $n = 20$.
\end{example}

For $n = 20 = 2^2 \cdot 5^1$, the graph $G_{Dp(20)}$ contains $D = 6$ vertices. By observing coprime adjacencies, we determine the vertex degrees: the central vertex $d(1) = 5$; the pure prime powers $d(2) = d(4) = 2$ and $d(5) = 3$; and the mixed products $d(10) = d(20) = 1$. The graph contains exactly $7$ edges (five connected to the center, plus $\{2,5\}$ and $\{4,5\}$).

\vspace{2mm}
\noindent \textbf{First Zagreb index ($M_1$):} \\
Manually summing the squared degrees yields:
$$M_1(G_{Dp(20)}) = 5^2 + 2(2^2) + 3^2 + 2(1^2) = 25 + 8 + 9 + 2 = 44.$$
Applying our generalized formula from Theorem \ref{thmzagreb} with $k_1=2$, $k_2=1$, and $D=6$ verifies this perfectly:
\begin{align*}
M_1(G_{Dp(20)}) &= \left[ ((2+1)^2 + 2) \cdot ((1+1)^2 + 1) \right] - 2(6) + 1\\
& = (11 \cdot 5) - 11 = 44.
\end{align*}

\vspace{2mm}
\noindent \textbf{Second Zagreb index ($M_2$):} \\
Manually summing the degree products of adjacent vertices gives $5(2+2+3+1+1) = 45$ for the central edges, and $2(3) + 2(3) = 12$ for the edges between pure powers:
$$M_2(G_{Dp(20)}) = 45 + 12 = 57.$$
Using our generalized formula from Theorem \ref{thmzagreb} confirms the combinatorial logic:
\begin{align*}
M_2(G_{Dp(20)})& = \frac{1}{2} \left[ 6(3(2)+1)(3(1)+1) - 2(2(2)+1)(2(1)+1) - 6^2 + 2(6) \right]\\
& = \frac{1}{2} \left[ 6(7)(4) - 2(5)(3) - 36 + 12 \right]\\
&= \frac{1}{2} [ 168 - 30 - 24 ] = 57.
\end{align*}

\noindent Hence, the theoretical formulas perfectly match the manual summations.

\begin{theorem}  \label{thmgut}Let $n = p_1^{k_1} p_2^{k_2} \dots p_r^{k_r}$ be the prime factorization of a positive integer $n$, where $p_i$ are distinct primes and $k_i \ge 1$. Let $D = \prod_{i=1}^r (k_i + 1)$ be the total number of divisors of $n$. The Gutman index of the divisor prime graph $G_{Dp(n)}$ is given by:
$$Gut(G_{Dp(n)}) = \left( \prod_{i=1}^r (2k_i + 1) - 1 \right)^2 - M_1(G_{Dp(n)}) - M_2(G_{Dp(n)}).$$
Where, $M_1(G_{Dp(n)})$ and $M_2(G_{Dp(n)})$ are the First and Second Zagreb indices of the graph.
\end{theorem}

\begin{proof}
Let $V$ be the vertex set of $G_{Dp(n)}$ with $|V| = D$. Because the central vertex $1$ is adjacent to all other vertices in the graph, the maximum shortest path distance between any two distinct vertices is $2$ from Lemma \ref{dia_lemma}. Thus, $d(u,v) = 1$ if $uv \in E$, and $d(u,v) = 2$ if $uv \notin E$.

The Gutman index $Gut(G)$ is defined as the sum of the degree products of all unordered pairs of distinct vertices, weighted by their shortest path distances:
$$Gut(G) = \sum_{\{u,v\} \subseteq V} d(u)d(v) d(u,v).$$
Partitioning this summation into edges and non-edges yields:
$$Gut(G) = \sum_{uv \in E} d(u)d(v)(1) + \sum_{uv \notin E} d(u)d(v)(2).$$
Recognizing that the sum over adjacent vertices is exactly the Second Zagreb index, $M_2(G)$, we simplify the expression to:
$$Gut(G) = M_2(G) + 2 \sum_{uv \notin E} d(u)d(v).$$
To evaluate the non-edge summation without direct counting, we utilize the standard algebraic identity for the sum of all cross-products: 
$$2 \sum_{u,v \in V} d(u)d(v) = \left(\sum_{v \in V} d(v)\right)^2 - \sum_{v \in V} d(v)^2.$$
Since $\sum_{v \in V} d(v)^2 = M_1(G)$ and the total sum of cross-products partitions into edges and non-edges, we can rewrite this identity as:
$$2 \left( \sum_{uv \in E} d(u)d(v) + \sum_{uv \notin E} d(u)d(v) \right) = \left(\sum_{v \in V} d(v)\right)^2 - M_1(G).$$
Substituting $M_2(G)$ for the edge sum and isolating the non-edge term gives:
$$2 \sum_{uv \notin E} d(u)d(v) = \left(\sum_{v \in V} d(v)\right)^2 - M_1(G) - 2M_2(G).$$
Substituting this result back into our partitioned Gutman equation yields a remarkably clean relationship:
$$Gut(G) = M_2(G) + \left[ \left(\sum_{v \in V} d(v)\right)^2 - M_1(G) - 2M_2(G) \right].$$
$$Gut(G) = \left(\sum_{v \in V} d(v)\right)^2 - M_1(G) - M_2(G).$$
Finally, we determine the total degree sum $\sum_{v \in V} d(v)$. Using the properties derived in previous theorems, the sum of the non-central degrees is $\prod_{i=1}^r (2k_i + 1) - D$, and the central degree is $d(1) = D - 1$. Therefore, the total degree sum is:
$$\sum_{v \in V} d(v) = (D - 1) + \left( \prod_{i=1}^r (2k_i + 1) - D \right) = \prod_{i=1}^r (2k_i + 1) - 1.$$
Squaring this sum and substituting it into the identity gives the final formula:
$$Gut(G_{Dp(n)}) = \left( \prod_{i=1}^r (2k_i + 1) - 1 \right)^2 - M_1(G_{Dp(n)}) - M_2(G_{Dp(n)}).$$
This completes the proof.
\end{proof}

\begin{example} 
Gutman index for the divisor prime graph of $n = 30$.
\end{example}

For $n = 30 = 2^1 \cdot 3^1 \cdot 5^1$, the graph $G_{Dp(30)}$ has $D = 8$ vertices. Based on their missing prime factors, the vertex degrees are: the central vertex $d(1) = 7$; the pure primes $d(2)=d(3)=d(5)=4$; the two-prime products $d(6)=d(10)=d(15)=2$; and the complete product $d(30)=1$. The total degree sum is $\sum_{v \in V} d(v) = 26$.

\vspace{2mm}
\noindent \textbf{Zagreb indices ($M_1$ and $M_2$):} \\
Summing the squared degrees yields the First Zagreb index:
$$M_1(G_{Dp(30)}) = 7^2 + 3(4^2) + 3(2^2) + 1^2 = 110.$$
Summing the degree products for all adjacent pairs yields the Second Zagreb index:
$$M_2(G_{Dp(30)}) = 133 \text{ (central)} + 48 \text{ (pure primes)} + 24 \text{ (mixed)} = 205.$$

\vspace{2mm}
\noindent \textbf{Gutman index ($Gut$):} \\
By definition, the Gutman index partitions into degree products weighted by distance $1$ (edges) and distance $2$ (non-edges). The sum of all pairwise degree products is $\frac{1}{2}(26^2 - 110) = 283$. The edge contribution is exactly $M_2 = 205$, leaving $283 - 205 = 78$ for the non-edges. Weighting these by distance gives:
$$Gut(G_{Dp(30)}) = 205(1) + 78(2) = 361.$$
Using our newly derived algebraic Theorem \ref{thmgut} confirms this instantly. With $k_1 = k_2 = k_3 = 1$:
\begin{align*}
 Gut(G_{Dp(30)}) &= \left( \prod_{i=1}^3 (2(1) + 1) - 1 \right)^2 - M_1(G_{Dp(30)}) - M_2(G_{Dp(30)})\\
& = (3 \cdot 3 \cdot 3 - 1)^2 - 110 - 205\\
&= (27 - 1)^2 - 315 = 26^2 - 315\\
& = 676 - 315 = 361.   
\end{align*}
The theoretical formula seamlessly confirms the manual structural calculation.

\begin{theorem}  \label{thmschultz}Let $n = p_1^{k_1} p_2^{k_2} \dots p_r^{k_r}$ be the prime factorization of a positive integer $n$, where $p_i$ are distinct primes and $k_i \ge 1$. Let $D = \prod_{i=1}^r (k_i + 1)$ be the total number of divisors of $n$. The Schultz index of the divisor prime graph $G_{Dp(n)}$ is given by:
$$S(G_{Dp(n)}) = 2(D - 1) \prod_{i=1}^r (2k_i + 1) - \prod_{i=1}^r \left( (k_i + 1)^2 + k_i \right) + 1.$$
\end{theorem}

\begin{proof}
Let $V$ be the vertex set of $G_{Dp(n)}$ with $|V| = D$. Because the central vertex $1$ is adjacent to all other vertices, the maximum shortest path distance between any two distinct vertices is $2$ from Lemma \ref{dia_lemma}. 

The Schultz index $S(G)$ evaluates the sum of the degrees of all unordered pairs of distinct vertices, weighted by their shortest path distances $d(u,v)$:
$$S(G) = \sum_{u,v \in V} (d(u) + d(v)) d(u,v).$$
Since $d(u,v) \in \{1, 2\}$, we partition this summation into edges (distance $1$) and non-edges (distance $2$):
$$S(G) = \sum_{uv \in E} (d(u) + d(v))(1) + \sum_{uv \notin E} (d(u) + d(v))(2).$$
By definition, the sum of the degree sums for adjacent vertices is exactly the First Zagreb index, $M_1(G)$. Thus, the equation simplifies to:
$$S(G) = M_1(G) + 2 \sum_{uv \notin E} (d(u) + d(v)).$$

To evaluate the sum over the non-edges without combinatorial counting, we first find the sum over all $\binom{D}{2}$ possible pairs. In a complete pairing of $D$ vertices, each vertex $u$ is paired with exactly $D - 1$ other vertices. Therefore, the degree $d(u)$ appears $D - 1$ times in the total sum:
$$\sum_{u,v \in V} (d(u) + d(v)) = (D - 1) \sum_{v \in V} d(v).$$
The sum over non-edges is simply the sum over all pairs minus the sum over the edges ($M_1(G)$):
$$\sum_{uv \notin E} (d(u) + d(v)) = (D - 1) \sum_{v \in V} d(v) - M_1(G).$$
Substituting this identity back into our partitioned Schultz equation yields a highly simplified relation:
$$S(G) = M_1(G) + 2 \left[ (D - 1) \sum_{v \in V} d(v) - M_1(G) \right]$$
$$S(G) = 2(D - 1) \sum_{v \in V} d(v) - M_1(G).$$
From our earlier derivations, we established the total degree sum as $\sum_{v \in V} d(v) = \prod_{i=1}^r (2k_i + 1) - 1$. Substituting this into the leading term gives:
$$2(D - 1) \sum_{v \in V} d(v) = 2(D - 1) \left( \prod_{i=1}^r (2k_i + 1) - 1 \right) = 2(D - 1) \prod_{i=1}^r (2k_i + 1) - 2D + 2.$$
Finally, we subtract the First Zagreb index, $M_1(G) = \prod_{i=1}^r \left( (k_i + 1)^2 + k_i \right) - 2D + 1$. 
Notice that the $-2D$ cross-terms perfectly cancel out:
$$S(G) = \left[ 2(D - 1) \prod_{i=1}^r (2k_i + 1) - 2D + 2 \right] - \left[ \prod_{i=1}^r \left( (k_i + 1)^2 + k_i \right) - 2D + 1 \right].$$
$$S(G_{Dp(n)}) = 2(D - 1) \prod_{i=1}^r (2k_i + 1) - \prod_{i=1}^r \left( (k_i + 1)^2 + k_i \right) + 1.$$
This completes the proof.
\end{proof}

\begin{example} 
Schultz index for the divisor prime graph of $n = 45$.
\end{example}

For $n = 45 = 3^2 \cdot 5^1$, the graph $G_{Dp(45)}$ contains $D = 6$ vertices. Based on their missing prime factors, the vertex degrees are: the central vertex $d(1) = 5$; the pure prime powers $d(5) = 3$ and $d(3) = d(9) = 2$; and the mixed products $d(15) = d(45) = 1$. The total degree sum is $\sum_{v \in V} d(v) = 14$.

\vspace{2mm}
\noindent \textbf{Schultz index ($S$):} \\
By definition, the Schultz index partitions into degree sums weighted by distance $1$ (edges) and distance $2$ (non-edges). The edge contribution is precisely the First Zagreb index:
$$M_1(G_{Dp(45)}) = \sum_{v \in V} d(v)^2 = 5^2 + 3^2 + 2(2^2) + 2(1^2) = 44.$$
As established in the theorem, the total sum of degrees for all $\binom{6}{2}$ pairs is $(D - 1)\sum d(v) = 5(14) = 70$. Subtracting the edge contribution leaves $70 - 44 = 26$ for the distance $2$ pairs. Weighting these by their distances yields:
$$S(G_{Dp(45)}) = 44(1) + 26(2) = 96.$$

\vspace{2mm}
\noindent \textbf{Verification using the generalized formula:} \\
Applying our finalized algebraic Theorem \ref{thmschultz} with $k_1 = 2$, $k_2 = 1$, and $D = 6$:
$$S(G_{Dp(n)}) = 2(D - 1) \prod_{i=1}^2 (2k_i + 1) - \prod_{i=1}^2 \left( (k_i + 1)^2 + k_i \right) + 1$$
$$S(G_{Dp(45)}) = 2(6 - 1) \left[ (2(2) + 1)(2(1) + 1) \right] - \left[ ((2+1)^2 + 2)((1+1)^2 + 1) \right] + 1 = 96.$$
The theoretical formula seamlessly confirms the structural calculation without requiring exhaustive edge enumeration.

\begin{theorem} \label{thmecc} Let $n = p_1^{k_1} p_2^{k_2} \dots p_r^{k_r}$ be the prime factorization of a positive integer $n$ (where $n$ is not prime, meaning $D \ge 3$), with distinct primes $p_i$ and $k_i \ge 1$. Let $D = \prod_{i=1}^r (k_i + 1)$ be the total number of divisors of $n$. The Eccentric Connectivity index of the divisor prime graph $G_{Dp(n)}$ is given by:
$$\xi^c(G_{Dp(n)}) = 2 \prod_{i=1}^r (2k_i + 1) - D - 1.$$
\end{theorem}

\begin{proof}
Let $V$ be the vertex set of $G_{Dp(n)}$, where $|V| = D$. The Eccentric Connectivity index $\xi^c(G)$ is defined as the sum of the products of the degree $d(v)$ and the eccentricity $\varepsilon(v)$ of each vertex:
$$\xi^c(G) = \sum_{v \in V} d(v) \varepsilon(v).$$
where $\varepsilon(v)$ is the maximum shortest path distance from $v$ to any other vertex. 

By Lemma \ref{dia_lemma}, the diameter of the graph is $2$. We evaluate the eccentricities by partitioning the vertices into the central vertex and the non-central vertices:
\begin{itemize}
    \item \textbf{Central vertex ($1$):} Because it connects to all other $D-1$ vertices, its maximum distance to any vertex is exactly $1$. Thus, $\varepsilon(1) = 1$, and its degree is $d(1) = D - 1$.
    \item \textbf{Non-central vertices ($v > 1$):} Because $n$ is not prime ($D \ge 3$), every non-central vertex shares a prime factor with at least one other vertex (or itself), meaning no non-central vertex connects to all of $V$. However, since all vertices connect to $1$, the maximum distance between $v$ and any non-adjacent vertex is $2$. Thus, $\varepsilon(v) = 2$ for all $v \in V \setminus \{1\}$.
\end{itemize}

Partitioning the index summation based on these eccentricities yields:
$$\xi^c(G) = d(1)\varepsilon(1) + \sum_{v > 1} d(v)\varepsilon(v)$$
$$\xi^c(G) = (D - 1)(1) + 2 \sum_{v > 1} d(v).$$

From our earlier derivations of the graph's properties, the sum of the degrees of all non-central vertices is already established as:
$$\sum_{v > 1} d(v) = \prod_{i=1}^r (2k_i + 1) - D.$$
Substituting this direct degree sum into our partitioned equation gives:
$$\xi^c(G) = (D - 1) + 2 \left( \prod_{i=1}^r (2k_i + 1) - D \right)$$
$$\xi^c(G) = D - 1 + 2 \prod_{i=1}^r (2k_i + 1) - 2D.$$
Combining the $D$ terms yields the final simplified formula:
$$\xi^c(G_{Dp(n)}) = 2 \prod_{i=1}^r (2k_i + 1) - D - 1.$$
This completes the proof.
\end{proof}

\begin{corollary} 
Let $p$ be a prime number and $k \ge 1$ be an integer. The Eccentric Connectivity index of the divisor prime graph $G_{Dp(p^k)}$ is $3k$ for $k \ge 2$, and $2$ for $k = 1$.
\end{corollary}

\begin{proof}
We evaluate the index in two cases based on the value of $k$:

\vspace{2mm}
\noindent \textbf{Case 1: $k \ge 2$} \\
When $k \ge 2$, the integer $n = p^k$ is a higher prime power. Because $n$ is not prime, it has $D = k + 1 \ge 3$ total divisors. Thus, it satisfies the conditions of the preceding theorem, and we can apply the generalized formula directly. 
Substituting $r = 1$, $k_1 = k$, and $D = k + 1$ yields:
\begin{align*}
    \xi^c(G_{Dp(p^k)}) &= 2 \prod_{i=1}^1 (2k_i + 1) - D - 1\\
& = 2(2k + 1) - (k + 1) - 1\\
& = 4k + 2 - k - 2 = 3k.
\end{align*}

\vspace{2mm}
\noindent \textbf{Case 2: $k = 1$} \\
When $k = 1$, $n = p$ is prime, and the total number of divisors is $D = 2$. The graph $G_{Dp(p)}$ is simply the complete graph $K_2$ with the vertex set $V = \{1, p\}$ connected by a single edge. 
Because both vertices share a maximum shortest path distance of $1$, their eccentricities are $\varepsilon(1) = \varepsilon(p) = 1$. With degrees $d(1) = d(p) = 1$, the index is:
$$\xi^c(G_{Dp(p)}) = d(1)\varepsilon(1) + d(p)\varepsilon(p) = (1 \times 1) + (1 \times 1) = 2.$$
This completes the proof.
\end{proof}

\begin{example} 
Eccentric Connectivity index for the divisor prime graph of $n = 22$.
\end{example}

For $n = 22 = 2^1 \cdot 11^1$, the graph $G_{Dp(22)}$ contains $D = 4$ vertices. Following the logic of our theorem, we partition the vertices by their eccentricities. The central vertex $1$ connects to all others, yielding $d(1) = 3$ and $\varepsilon(1) = 1$. The non-central vertices $2, 11$, and $22$ have degrees $d(2)=2$, $d(11)=2$, and $d(22)=1$, respectively. Because they share prime factors but all connect to the center, their eccentricities are uniformly $\varepsilon(v) = 2$.

\vspace{2mm}
\noindent \textbf{Eccentric Connectivity index ($\xi^c$):} \\
Manually summing the products of the degrees and eccentricities yields:
$$\xi^c(G_{Dp(22)}) = d(1)\varepsilon(1) + \sum_{v > 1} d(v)\varepsilon(v)$$
$$\xi^c(G_{Dp(22)}) = 3(1) + \left[ 2(2) + 2(2) + 1(2) \right] = 3 + 10 = 13.$$

Applying our generalized formula from Theorem \ref{thmecc} with $k_1 = 1$, $k_2 = 1$, and $D = 4$ confirms this result instantly:
\begin{align*}
    \xi^c(G_{Dp(22)})& = 2 \prod_{i=1}^2 (2k_i + 1) - D - 1\\
& = 2(2(1)+1)(2(1)+1) - 4 - 1\\
& = 2(3 \cdot 3) - 5 = 18 - 5 = 13.
\end{align*}

The theoretical formula seamlessly confirms the manual structural calculation.

\section*{Acknowledgements}
The First author acknowledges the financial support received from CSIR under the CSIR-Junior Research Fellowship (JRF) scheme.

\end{document}